\documentclass{amsart}

\newtheorem{theorem}{Theorem}[section]

\newtheorem{proposition}[theorem]{Proposition}
\newtheorem{corollary}[theorem]{Corollary}

\theoremstyle{definition}

\theoremstyle{remark}

\numberwithin{equation}{section}

\usepackage[pagebackref,colorlinks=true]{hyperref}  

\newcommand{\f}{\varphi}
\newcommand{\g}{\tilde{g}}
\newcommand{\n}{\nabla}
\newcommand{\ttt}{\tilde\tau}

\newcommand{\M}{(M,\A\f,\A\xi,\A\eta,\A{}g)}

\newcommand{\I}{\iota}

\newcommand{\R}{\mathbb R}

\newcommand{\F}{\mathcal{F}}
\newcommand{\LL}{\mathcal{L}}
\newcommand{\ta}{\theta}
\newcommand{\om}{\omega}
\newcommand{\lm}{\lambda}

\newcommand{\al}{\alpha}
\newcommand{\bt}{\beta}

\newcommand{\D}{\mathrm{d}\hspace{-0.5pt}}

\DeclareMathOperator{\Span}{span} 

\newcommand{\A}{\allowbreak{}}

\newcommand{\thmref}[1]{Theorem~\ref{#1}}

\newcommand{\propref}[1]{Proposition~\ref{#1}}

\begin{document}

\title[Ricci-like solitons with vertical potential ...]
{Ricci-like solitons with vertical potential 
on Sasaki-like almost contact B-metric manifolds}


\author{Mancho Manev}
\address
{University of Plovdiv Paisii Hilendarski,
Faculty of Mathematics and Informatics, Department of Algebra and
Geometry, 24 Tzar Asen St., Plovdiv 4000, Bulgaria
\&
Medical University of Plovdiv, Faculty of Public Health,
Department of Medical Informatics, Biostatistics and E-Learning, 15A Vasil Aprilov Blvd.,
Plovdiv 4002, Bulgaria} 
\email{mmanev@uni-plovdiv.bg}



\begin{abstract}
Ricci-like solitons on Sasaki-like almost contact B-metric manifolds are the object of study. 
Cases, where the potential of the Ricci-like soliton is the Reeb vector field or pointwise collinear to it, are considered.
In the former case, 
the properties for a parallel or recurrent Ricci-tensor are studied. 
In the latter case, 
it is shown that the potential of the considered Ricci-like soliton has a constant length and the manifold is $\eta$-Einstein. Other curvature conditions are also found, which imply that the main metric is Einstein. 
After that, some results are obtained for a parallel symmetric second-order covariant tensor on the manifolds under study.
Finally, an explicit example of dimension 5 is given and some of the results are illustrated.
\end{abstract}




%
%


\maketitle

\section*{Introduction}

Following the R.\,S. Hamilton's concept of a Ricci flow, introduced in 1982 \cite{Ham82},  
R. Sharma initiated the study of Ricci solitons in contact Riemannian geometry, particularly on K-contact manifolds in 2008  \cite{Shar}. 
After that
Ricci solitons have been studied on different kinds of almost contact metric manifolds, e.g.
$\al$-Sasakian \cite{IngBag}, trans-Sasakian \cite{BagIng13}, Kenmotsu
\cite{NagPre}, 
etc. 

Although this topic was first studied in Riemannian geometry, 
in recent years Ricci solitons and their generalizations 
have also been studied in pseudo-Riemannian manifolds, 
mostly with Lorentzian metrics 
(\cite{BagIng12}, \cite{Bla16}, 
\cite{BlaPerAceErd}, \cite{Mat}, \cite{YadKusNar19}).

Almost contact metric manifolds are usually equipped with a compatible metric, 
which can be Riemannian or pseudo-Riemannian. 
Then, the contact $(1,1)$-tensor $\f$ acts as an isometry with respect to the metric 
on the contact distribution. 
The alternative case is when this action is an anti-isometry. 
Then, the metric is necessarily pseudo-Riemannian and it is known as a B-metric. 
Moreover, the associated (0,2)-tensor of a B-metric with respect to $\f$ is also a B-metric, 
in contrast to the associated (0,2)-tensor of a compatible metric is a 2-form.
The differential geometry of almost contact B-metric manifolds has been studied since 1993
\cite{GaMiGr}, \cite{ManGri93}. 

In our previous  work \cite{Man62}, we introduced the notion of Ricci-like solitons on almost contact B-metric manifolds as a generalization of Ricci solitons and $\eta$-Ricci solitons using both B-metrics of manifold. 
There, we investigated the case of a potential Reeb vector field, when either the manifold is  Sasaki-like or the Reeb vector field is torse-forming.

In the present paper, our goal is to continue our study on Ricci-like solitons, a generalization of Ricci solitons 
compatible with the almost contact B-metric structure. 
Now we focus on the case, when the potential of the considered Ricci-like soliton is a vertical vector field, 
i.e. it is orthogonal to the contact distribution with respect to the metric. 
The paper is organized as follows.  
Section 1 is devoted to the basic concepts of almost contact B-metric manifolds of Sasaki-like type.
In Section 2, we study the covariant derivative of the Ricci tensor with respect to the main B-metric $g$ of the manifold with a Ricci-like soliton, which potential is exactly the Reeb vector field $\xi$.   
In Section 3, we consider Ricci-like solitons with potential, which is the Reeb vector field multiplied by a differentiable function $k$. Then, we prove that $k$ is constant and the Sasaki-like manifold is $\eta$-Einstein. Also a series of curvature conditions of the manifold are considered that imply $k$ is a constant and $g$ is an Einstein metric. 
In Section 4, we give some characterization for Ricci-like solitons on Sasaki-like manifold 
concerning a parallel symmetric $(0,2)$-tensor.  
In Section 5, we give an explicit example of a Lie group of dimension 5 equipped with a Sasaki-like almost contact B-metric structure. Then, we show that the manifold is $\eta$-Einstein of a special type, 
admits a Ricci-like soliton with potential $\xi$ and 
the results for this example support the relevant assertions in the previous sections.

\section{Sasaki-like almost contact B-metric manifolds}


We consider \emph{almost contact B-metric manifolds}. A differentiable manifold $M$ of this type has dimension $(2n+1)$ and it is denoted by $\M$, where  $(\f,\xi,\eta)$ is an almost
contact structure and $g$ is a B-metric. More precisely, $\f$ is an endomorphism
of the tangent bundle $TM$, $\xi$ is a Reeb vector field, $\eta$ is its dual contact 1-form and
$g$ is a pseu\-do-Rie\-mannian
metric $g$  of signature $(n+1,n)$ satisfying the following conditions \cite{GaMiGr}
\begin{equation}\label{strM}
\begin{array}{c}
\f\xi = 0,\qquad \f^2 = -\I + \eta \otimes \xi,\qquad
\eta\circ\f=0,\qquad \eta(\xi)=1,\\[4pt]
g(\f x, \f y) = - g(x,y) + \eta(x)\eta(y),
\end{array}
\end{equation}
where $\I
$ is the identity transformation on $\Gamma(TM)$.

In the latter equality and further, $x$, $y$, $z$, $w$ will stand for arbitrary elements of $\Gamma(TM)$ or vectors in the tangent space $T_pM$ of $M$ at an arbitrary
point $p$ in $M$.

Some immediate consequences of \eqref{strM} are the following equations
\begin{equation}\label{strM2}
\begin{array}{ll}
g(\f x, y) = g(x,\f y),\qquad &g(x, \xi) = \eta(x),
\\[4pt]
g(\xi, \xi) = 1,\qquad &\eta(\n_x \xi) = 0,
\end{array}
\end{equation}
where $\n$ is the Levi-Civita connection of $g$.

The associated metric $\g$ of $g$ on $M$ is defined by
\begin{equation}\label{gg}
\g(x,y)=g(x,\f y)+\eta(x)\eta(y).
\end{equation}
It is also a B-metric.

A classification of almost contact B-metric manifolds, consisting of eleven basic classes $\F_i$, $i\in\{1,2,\dots,11\}$, is given in
\cite{GaMiGr}. This classification is made with respect
to the (0,3)-tensor $F$ defined by
\begin{equation}\label{F=nfi}
F(x,y,z)=g\bigl( \left( \nabla_x \f \right)y,z\bigr).
\end{equation}
It possess the following basic properties:
\begin{equation}\label{F-prop}
\begin{array}{l}
F(x,y,z)=F(x,z,y)
=F(x,\f y,\f z)+\eta(y)F(x,\xi,z)
+\eta(z)F(x,y,\xi),\\[4pt]
F(x,\f y, \xi)=(\n_x\eta)y=g(\n_x\xi,y).
\end{array}
\end{equation}

The intersection of the basic classes is the special class $\F_0$,
determined by the condition $F=0$, and it is known as the
class of the \emph{cosymplectic B-metric manifolds}.

The Lee forms of $\M$ are the following 1-forms
associated with $F$:
\begin{equation*}\label{t}
\theta(z)=g^{ij}F(e_i,e_j,z),\quad
\theta^*(z)=g^{ij}F(e_i,\f e_j,z), \quad \omega(z)=F(\xi,\xi,z),
\end{equation*}
where $\left(g^{ij}\right)$ is the inverse matrix of the
matrix $\left(g_{ij}\right)$ of $g$ with respect to a basis $\left\{e_i;\xi\right\}$ $(i=1,2,\dots,2n)$ of
$T_pM$.  Obviously,
$\om(\xi)=\ta^*\circ\f+\ta\circ\f^2=0$ are valid.


In \cite{IvMaMa45}, it is introduced the type of \emph{Sasaki-like} manifolds among almost
contact B-metric manifolds (also known as almost contact complex Riemannian manifolds). The definition condition is its complex cone to be a K\"ahler-Norden manifold, i.e. with a parallel complex structure.
A Sasaki-like manifold with almost
contact B-metric structure is determined by the condition
\begin{equation}\label{defSl}
\begin{array}{l}
\left(\nabla_x\f\right)y=-g(x,y)\xi-\eta(y)x+2\eta(x)\eta(y)\xi,
\end{array}
\end{equation}
or equivalently $\left(\nabla_x\f\right)y=g(\f x,\f y)\xi+\eta(y)\f^2 x$.

Obviously, any Sasaki-like manifold belongs to the class $\F_4$ and its Lee forms are $\ta=-2n\,\eta$, $\ta^*=\om=0$.
Moreover,
the following identities are valid for it \cite{IvMaMa45}
\begin{equation}\label{curSl}
\begin{array}{ll}
\n_x \xi=-\f x, \qquad &\left(\n_x \eta \right)(y)=-g(x,\f y),\\[4pt]
R(x,y)\xi=\eta(y)x-\eta(x)y, \qquad &\rho(x,\xi)=2n\, \eta(x), \\[4pt]
R(\xi,y)z=g(y,z)\xi-\eta(z)y,\qquad 				&\rho(\xi,\xi)=2n,
\end{array}
\end{equation}
where $R$ and $\rho$ stand for the curvature tensor and the Ricci tensor of $\n$.

Let $\tau$ and $\tilde\tau$ be the scalar curvatures with respect to $g$ and $\g$, respectively, 
and let $\tau^*$ be the associated quantity  of $\tau$ regarding $\f$, defined by $\tau^*=g^{ij}\rho(e_i,\f e_j)$.

For an arbitrary $\F_4$-manifold with a closed 1-form $\ta$, 
the following relation for $\tau^*$ and $\ttt$ is given in \cite{Man3} 
\[
\tilde\tau=-\tau^*+	\frac{1}{2n}(\ta(\xi))^2+\xi(\ta(\xi)).
\] 
Then, for a Sasaki-like manifold since $\ta(\xi)=-2n$ we have
\[
\tilde\tau=-\tau^*+2n.
\]

Let us recall  \cite{Man62}, an almost contact B-metric manifold $\M$ is said to be
\emph{Einstein-like} if its Ricci tensor $\rho$ satisfies
\begin{equation}\label{defEl}
\begin{array}{l}
\rho=a\,g +b\,\g +c\,\eta\otimes \eta
\end{array}
\end{equation}
for some triplet of constants $(a,b,c)$ \cite{Man62}.
In particular, when $b=0$ and $b=c=0$, the manifold is called an \emph{$\eta$-Einstein manifold} and an \emph{Einstein manifold}, respectively.

Tracing \eqref{defEl}, the scalar curvature $\tau$ of an Einstein-like almost contact B-metric manifold has the form
\begin{equation}\label{tauEl}
\tau=(2n+1)a+b+c.
\end{equation}

For a Sasaki-like manifold  $\M$ with $\dim M = 2n+1$ and a scalar curvature $\tau$ regarding $g$, which is Einstein-like with a triplet of
constants $(a,b,c)$, the following equalities are given in \cite{Man62}:
\begin{equation}\label{abctau-ElSl}
a + b + c = 2n, \qquad \tau = 2n(a + 1). 
\end{equation}
Then, for $\tilde\tau$ on an Einstein-like Sasaki-like manifold we obtain
\begin{equation}\label{abctau*-ElSl}
\tilde\tau=2n(b+1)
\end{equation}
and because \eqref{tauEl}, \eqref{abctau-ElSl} and \eqref{abctau*-ElSl}, the expression \eqref{defEl} becomes
\begin{equation*}\label{defElSl}
\begin{array}{l}
\rho=\left(\dfrac{\tau}{2n}-1\right)g +\left(\dfrac{\tilde\tau}{2n}-1\right)\g 
+\left(2(n+1)-\dfrac{\tau+\tilde\tau}{2n}\right)\eta\otimes \eta.
\end{array}
\end{equation*}

\begin{proposition}[\cite{Man64}]\label{prop:El-Dtau}
Let $\M$ be a $(2n+1)$-dimensional Sasaki-like manifold. If it is almost Einstein-like  with functions $(a,b,c)$
then the scalar curvatures $\tau$ and $\tilde\tau$ of $g$ and $\g$, respectively, are constants 
\begin{equation*}\label{El-Dtauxi}
\tau = const, \qquad \ttt=2n
\end{equation*}
and $\M$ is $\eta$-Einstein with constants  
\[
(a,b,c)=\left(\frac{\tau}{2n}-1,\,0,\,2n+1-\frac{\tau}{2n}\right).
\]
\end{proposition}

\section{Ricci-like solitons with potential Reeb vector field on Sasaki-like manifolds}

In \cite{Man62}, by conditions for Ricci tensor, there are introduced the notions of an Einstein-like almost contact B-metric manifold and a Ricci-like soliton on an almost contact B-metric manifold.

It is said that $\M$ admits a \emph{Ricci-like soliton with potential $\xi$}
if the following condition is satisfied for a triplet of constants $(\lm,\mu,\nu)$
\begin{equation}\label{defRl}
\begin{array}{l}
\frac12 \mathcal{L}_{\xi} g  + \rho + \lm\, g  + \mu\, \g  + \nu\, \eta\otimes \eta =0,
\end{array}
\end{equation}
where $\mathcal{L}$ denotes the Lie derivative.

If $\mu=0$ (respectively, $\mu=\nu=0$), then \eqref{defRl} defines an \emph{$\eta$-Ricci soliton} 
(respectively, a \emph{Ricci soliton}) on $\M$. 
 
If $\lm$, $\mu$, $\nu$ are functions on $M$, then the soliton is called \emph{almost Ricci-like soliton}, \emph{almost $\eta$-Ricci soliton} and \emph{almost Ricci soliton}, respectively.

%


\begin{theorem}[\cite{Man62}]\label{thm:RlSl}
Let $\M$ be a $(2n+1)$-dimensional Sasaki-like manifold and let $a$, $b$, $c$, $\lm$, $\mu$, $\nu$ be constants that satisfy the following equalities:
\begin{equation}\label{SlElRl-const}
a+\lm=0,\qquad b+\mu-1=0,\qquad c+\nu+1=0.
\end{equation}
Then, the manifold admits
a Ricci-like soliton with potential $\xi$ and constants $(\lm,\A\mu,\A\nu)$, where $\lm+\mu+\nu=-2n$,
if and only if
it is Einstein-like with constants $(a,b,c)$, where $a+b+c=2n$.

In particular, we get:
\begin{enumerate}
	\item[(i)]    The manifold admits an $\eta$-Ricci soliton with potential $\xi$ and constants $(\lm,0,-2n-\lm)$ if and only if
the manifold is Einstein-like with constants $(-\lm,1,\lm+2n-1)$.

	\item[(ii)]   The manifold admits a shrinking Ricci soliton with potential $\xi$ and constant $-2n$ if and only if
the manifold is Einstein-like with constants $(2n,1,-1)$.

	\item[(iii)]   The manifold is $\eta$-Einstein with constants $(a,0,2n-a)$  if and only if
it admits a Ricci-like soliton with potential $\xi$ and constants $(-a,1,a-2n-1)$.

	\item[(iv)]   The manifold is Einstein with constant $2n$ if and only if
it admits a Ricci-like soliton with potential $\xi$ and constants $(-2n,1,-1)$.
\end{enumerate}
\end{theorem}

If $\M$ is Sasaki-like, we have
\[
\left(\mathcal{L}_{\xi} g\right)(x,y)=g(\n_x\xi,y)+g(x,\n_y\xi)=-2g(x,\f y),
\]
i.e.
$\frac12 \mathcal{L}_{\xi} g=-\g+\eta\otimes \eta$.  Then, because of \eqref{defRl}, $\rho$ takes the form
\begin{equation}\label{SlRl-rho}
	\rho = -\lm g + (1-\mu) \g  - (1+\nu) \eta\otimes \eta.
\end{equation}

By direct computations from \eqref{SlRl-rho} we infer the following
\begin{corollary}
Let $\M$ satisfy the conditions in the general case of \thmref{thm:RlSl}.
Then, the constants $a$, $b$, $c$, $\lm$, $\mu$, $\nu$ are expressed by scalar curvatures $\tau$ and $\tilde\tau$
with respect to $g$ and $\g$, respectively, as follows
\[
\begin{array}{lll}
\lm=1-\frac{1}{2n}\tau,\qquad &\mu=2-\frac{1}{2n}\tilde\tau, \qquad &\nu=\frac{1}{2n}(\tau+\tilde\tau)-2n-3,
\\[4pt]
a=\frac{1}{2n}\tau-1,\qquad &b=\frac{1}{2n}\tilde\tau-1, \qquad &c=2n+2 -\frac{1}{2n}(\tau+\tilde\tau).
\end{array}
\]
\end{corollary}

Using \eqref{gg}, \eqref{F=nfi},  \eqref{F-prop}, \eqref{defSl} and \eqref{SlElRl-const},
we apply covariant derivatives to \eqref{SlRl-rho} and we get
\begin{equation}\label{SlRl-nrho}
\begin{array}{l}
\left(\n_x \rho\right)(y,z)=(1-\mu) \{ g(\f x, \f y)\eta(z)+g(\f x, \f z)\eta(y)\}\\[4pt]
\phantom{\left(\n_x \rho\right)(y,z)}
+(\mu+\nu)\{ g(x, \f y)\eta(z)+g(x, \f z)\eta(y)\}.
\end{array}
\end{equation}
Obviously, the tensors
\(
\left(\n_x \rho\right)(\f y,\f z)\),  
\(\left(\n_{\xi} \rho\right)(y,z)\) and 
\(\left(\n_{x} \rho\right)(\xi,\xi)\) vanish and therefore we establish the truthfulness of the following

\begin{proposition}\label{prop:rho-Sl}
Every Einstein-like Sasaki-like manifold $\M$ admitting
a Ricci-like soliton with potential $\xi$ 
is Ricci $\eta$-parallel and Ricci parallel along $\xi$,
i.e. $(\n\rho)|_{\ker\eta}=0$ and $\n_{\xi}\rho=0$, respectively.
\end{proposition}

\begin{proposition}\label{prop:nrho-Sl}
Let $\M$ be a 
Sasaki-like manifold admitting
a Ricci-like soliton with potential $\xi$ and constants $(\lm,\mu,\nu)$.
The manifold is locally Ricci symmetric if and only if $(\lm,\mu,\nu)=(-2n,1,-1)$, i.e. it is an Einstein manifold.
\end{proposition}
\begin{proof}
The manifold is locally Ricci symmetric, i.e. $\left(\n_x \rho\right)(y,z)=0$, if and only if $1-\mu=\mu+\nu=0$, which is equivalent to $\mu=-\nu=1$. The value of $\lm$ comes from the condition $\lm+\mu+\nu=-2n$ since the manifold is
Sasaki-like. 
The conclusion that the manifold is Einstein follows from \thmref{thm:RlSl} (iv). 
\end{proof}

Let us remark that the Sasaki-like manifold is locally Ricci symmetric just in the case (iv) of \thmref{thm:RlSl}.

Now for the Ricci tensor, we consider the term \emph{$\n$-recurrent}, which is weaker than the usual parallelism.
We say that the Ricci tensor is $\n$-recurrent if its covariant derivative with respect to $\n$, i.e. $\n\rho$, is expressed only by $\rho$ and some 1-form.

\begin{proposition}\label{prop:nrho=rho-Sl}
Let $\M$ be a 
Sasaki-like manifold admitting
a Ricci-like soliton with potential $\xi$ and constants $(\lm,\mu,\nu)$ with the condition $(\lm,\mu)\neq(0,1)$.
Then the Ricci tensor $\rho$ of this manifold is $\n$-recurrent and satisfies the formula
\begin{equation}\label{SlRl-nrho2}
\begin{array}{l}
\left(\n_x \rho\right)(y,z)=\dfrac{(1-\mu)^2+\lm(\lm+2n)}{\lm^2+(1-\mu)^2}\{\rho( x, \f y)\eta(z)+ \rho( x, \f z)\eta(y)\}\\[4pt]
\phantom{\left(\n_x \rho\right)(y,z)=}
-\dfrac{2(\lm+n)(1-\mu)}{\lm^2+(1-\mu)^2}\{ \rho(\f x, \f y)\eta(z)+ \rho(\f x, \f z)\eta(y)\}.
\end{array}
\end{equation}
\end{proposition}
\begin{proof}
The equality \eqref{SlRl-rho}, by virtue of \eqref{strM}, \eqref{strM2}, \eqref{gg} and $\lm+\mu+\nu=-2n$,
can be rewritten as
\[
	\rho(x,y) = \lm\, g(\f x,\f y) + (1-\mu) g(x,\f y)  +2n\, \eta(x)\eta(y)
\]
and therefore there are valid the following two equalities
\[
\begin{array}{l}
	\rho(x,\f y) = -\lm\, g(x,\f y) + (1-\mu) g(\f x,\f y),	
\\[4pt]
	\rho(\f x,\f y) = -\lm\, g(\f x,\f y) - (1-\mu) g(x,\f y).	
\end{array}
\]
The system of the latter two equations for $(\lm,\mu)\neq(0,1)$ is solved with respect to $g(\f x,\f y)$ and $g(x,\f y)$ as follows
\[
\begin{array}{l}
	g(x,\f y) = \dfrac{1}{\lm^2+(1-\mu)^2}\{-\lm\,\rho(x,\f y) + (1-\mu)\,\rho(\f x,\f y)\},	
\\[4pt]
	g(\f x,\f y) = \dfrac{1}{\lm^2+(1-\mu)^2}\{-\lm\rho(\f x,\f y) + (1-\mu) \rho(x,\f y)\}.	
\end{array}
\]
Substituting the latter equalities to \eqref{SlRl-nrho}, we get the recurrent dependence \eqref{SlRl-nrho2} of the Ricci tensor.
\end{proof}

\section{Ricci-like solitons with potential pointwise collinear with the Reeb vector field on Sasaki-like manifolds}

Similarly to the definition of a Ricci-like soliton with potential $\xi$ by \eqref{defRl},
we introduce the following notion.
We say that $\M$ admits a \emph{Ricci-like soliton with potential vector field $v$}
if the following condition is satisfied for a triplet of constants $(\lm,\mu,\nu)$
\begin{equation}\label{defRl-v}
\begin{array}{l}
\frac12 \mathcal{L}_{v} g  + \rho + \lm\, g  + \mu\, \g  + \nu\, \eta\otimes \eta =0.
\end{array}
\end{equation}

Suppose that $\M$ is a Sasaki-like manifold admitting a Ricci-like soliton whose potential vector field $v$ is pointwise collinear with $\xi$, i.e. $v=k\,\xi$, where $k$ is a differentiable function on $M$.
It is clear that $k=\eta(v)$ and therefore $v$ belongs to the vertical distribution $H^\bot=\Span\xi$, which is orthogonal to the contact distribution $H=\ker\eta$ with respect to $g$.

\begin{theorem}\label{thm:k=const}
Let $\M$, $\dim{M}=2n+1$, be a Sasaki-like manifold admitting a Ricci-like soliton with constants $(\lm,\mu,\nu)$ whose potential vector field $v$ is pointwise collin\-ear with the Reeb vector field $\xi$, i.e. $v=k\,\xi$, where $k$ is a differentiable function on $M$.
Then we have $k=\mu$, i.e. $k$ is constant, the equality $\lm+\nu=-k-2n$ is valid and the manifold is $\eta$-Einstein with constants
\[
(a,b,c)=(-\lm,0,\lm+2n).
\]
\end{theorem}
\begin{proof}
Due to the first equality in \eqref{curSl}, the expression of $\LL_v g$ in the considered case has the form
\[
\begin{array}{l}
\left(\LL_v g\right)(x,y)=g(\n_x v,y)+g(x,\n_y v)=g(\n_x k\xi,y)+g(x,\n_y k\xi)\\[4pt]
\phantom{\left(\LL_v g\right)(x,y)}
=\D{k}(x)\eta(y)+\D{k}(y)\eta(x)-2kg(x,\f y).
\end{array}
\]
Replacing it in \eqref{defRl-v}, we have
\begin{equation}\label{SlRl-v}
\begin{array}{l}
\D{k}(x)\eta(y)+\D{k}(y)\eta(x)=-2\{\rho(x,y)+\lm g(x,y)-(k-\mu)g(x,\f y) \\[4pt]
\phantom{\D{k}(x)\eta(y)+\D{k}(y)\eta(x)=-2\{\rho(x,y)}
+(\mu+\nu)\eta(x)\eta(y)\}.
\end{array}
\end{equation}
Substituting $y$ for $\xi$ and using the expression of $\rho(x,\xi)$  from \eqref{curSl}, the latter equality implies
\begin{equation}\label{SlRl-v-xk}
\begin{array}{l}
\D{k}(x)= -\{\D{k}(\xi)+2(\lm+\mu+\nu+2n)\}\eta(x),
\end{array}
\end{equation}
that yields the following equality for $x=\xi$
\begin{equation*}\label{SlRl-v-xik}
\begin{array}{l}
\D{k}(\xi)= -(\lm+\mu+\nu+2n).
\end{array}
\end{equation*}
Therefore, \eqref{SlRl-v-xk} takes the form
\begin{equation}\label{SlRl-v-xk2}
\begin{array}{l}
\D{k}(x)= -(\lm+\mu+\nu+2n)\eta(x).
\end{array}
\end{equation}

Bearing in mind \eqref{SlRl-v-xk2} and \eqref{SlRl-v}, we obtain the following expression of the Ricci tensor
\begin{equation}\label{SlRl-v-rho}
\rho=-\lm g +(k-\mu) \g +(\lm+\mu+2n-k)\eta\otimes\eta.
\end{equation}

The latter equality is the condition for manifold to be almost Einstein-like with functions 
\[
(a,b,c)=(-\lm,k-\mu,\lm+\mu+2n-k).
\]
%
Then, according to \propref{prop:El-Dtau}, $\M$ is $\eta$-Einstein with constants 
\[
(a,b,c)=\left(\frac{\tau}{2n}-1,\,0,\,2n+1-\frac{\tau}{2n}\right).
\]
Comparing the two triads $(a,b,c)$ from the above, we infer that $k=\mu$, i.e. $k$ is a constant.

Thus, from the formula in \eqref{SlRl-v-xk2} we deduce that the condition 
\begin{equation*}\label{lmn2n}
\lm+\mu+\nu=-2n
\end{equation*}
is satisfied.
Then, equality \eqref{SlRl-v-rho} becomes
\begin{equation}\label{SlRl-v-rho-k=const}
\rho=-\lm g +(\lm+2n)\eta\otimes\eta,
\end{equation}
which completes the proof.
\end{proof}

%

\subsection{Additional curvature properties}

Let $\M$, $\dim{M}=2n+1$, be a Sasaki-like manifold admitting a Ricci-like soliton with vertical potential $v$, i.e. $v=k\,\xi$, $k=const$. Then, according to \thmref{thm:k=const}, the soliton constants  are $(\lm,k,-\lm -k -2n)$
and $\M$ is $\eta$-Einstein with constants $(a,b,c)=(-\lm,0,\lm+2n)$, where $\lm=1-\frac{1}{2n}\tau$.


If consider the condition for \emph{locally Ricci symmetry}, i.e. $\n\rho$ vanishes, we immediately obtain the following 
\begin{proposition}\label{prop:nrho-SletaEl}
Let $\M$ be a 
Sasaki-like manifold admitting
a Ricci-like soliton with vertical potential. 
The manifold is locally Ricci symmetric if and only if it is an Einstein manifold.
\end{proposition}
\begin{proof}
By virtue of 
\eqref{SlRl-v-rho-k=const}, in similar way as for \eqref{SlRl-nrho}, we obtain in this case
\begin{equation}\label{SlRl-nrho-k}
\begin{array}{l}
\left(\n_x \rho\right)(y,z)=-(\lm+2n)\{ g(x, \f y)\eta(z)+g(x, \f z)\eta(y)\}.
\end{array}
\end{equation}
Then, the statement is easy to conclude.
\end{proof}


Let us impose the condition $R(\xi,\ldotp)\cdot\rho=0$, i.e. it is valid the following
\begin{equation}\label{Rrho}
\rho\left(R(\xi,x)y,z\right)+\rho\left(y,R(\xi,x)z\right)=0.
\end{equation}

\begin{proposition}\label{prop:2-SletaEl}
Let $\M$ be a 
Sasaki-like manifold admitting
a Ricci-like soliton with vertical potential. 
The condition $R(\xi,\ldotp)\cdot\rho=0$ is satisfied if and only if the manifold  is Einstein.
\end{proposition}
\begin{proof}
Then, because of \eqref{curSl}, equality \eqref{Rrho} follows
\begin{equation}\label{rhorho}
\rho(x,y)\eta(z)+\rho(x,z)\eta(y)-2n g(x,y)\eta(z)-2n g(x,z)\eta(y)=0.
\end{equation}

By virtue of \eqref{SlRl-v-rho-k=const} and \eqref{rhorho}, we obtain
\begin{equation}\label{lgff}
\begin{array}{l}
(\lm+2n)\left\{g(\f x,\f y)\eta(z)
+g(\f x,\f z)\eta(y)\right\}=0.
\end{array}
\end{equation}

Equality \eqref{lgff} for $z=\xi$ provides
$
(\lm+2n)g(\f x,\f y)=0,
$
which is equivalent to 
\begin{equation*}\label{lmk2n}
\lm=-2n.
\end{equation*}
Therefore, condition \eqref{Rrho} implies that the manifold is Einstein.
\end{proof}


A non-vanishing Ricci tensor $\rho$ is called \emph{cyclic parallel} or  \emph{of Codazzi type} 
if it satisfies the condition \cite{Gray78}
\[
(\n_x \rho)(y,z) + (\n_y \rho)(z,x) + (\n_z \rho)(x,y) =0
\]
or
\[
(\n_x \rho)(y,z) = (\n_y \rho)(x,z),
\]
respectively.

It is said that a non-vanishing Ricci operator $Q$ is \emph{Ricci $\f$-symmetric} 
if it satisfies the condition \cite{DeSa08}
\[
\f^2(\n_x Q)y =0.
\]
In \cite{GhoDe17}, if the latter condition is satisfied for arbitrary vector fields on the manifold, it is called  
\emph{globally Ricci $\f$-symmetric}, and when these vector fields are orthogonal to $\xi$, then the manifold is called 
\emph{locally Ricci $\f$-symmetric}.

Similarly to \propref{prop:2-SletaEl}, we establish the truthfulness of the following 
\begin{proposition}\label{prop:3-SletaEl}
Let $\M$ be a 
Sasaki-like manifold admitting
a Ricci-like soliton with vertical potential. 
\begin{enumerate}
	\item[(i)] The manifold has a cyclic parallel  Ricci tensor if and only if it is Einstein.
	\item[(ii)] The Ricci tensor is of Codazzi type if and only the manifold is Einstein.
	\item[(iii)] The manifold is locally Ricci $\f$-symmetric.
	\item[(iv)] The manifold is globally Ricci $\f$-symmetric  if and only if it is Einstein.
\end{enumerate}
\end{proposition}

In \cite{ChaKaw07}, the notion of \emph{almost pseudo Ricci symmetric manifolds} is  introduced 
by the following condition for its non-vanishing Ricci tensor 
\begin{equation}\label{apRs-def}
(\n_x \rho)(y,z) = \{\al(x) + \bt(x)\}\rho(y,z) + \al(y)\rho(x,z) + \al(z)\rho(x,z),
\end{equation}
where $\al$ and $\bt$ are non-vanishing 1-forms.

\begin{proposition}\label{prop:6-SletaEl}
Let $\M$ be a Sasaki-like manifold admitting
a Ricci-like soliton with vertical potential.
Then the manifold is almost pseudo Ricci symmetric 
if and only if the manifold is Einstein.
\end{proposition}
\begin{proof}
Substituting  \eqref{SlRl-nrho-k} in the defining condition \eqref{apRs-def} for an almost pseudo Ricci symmetric manifold gives us the following equality
\begin{equation}\label{apRs}
\begin{array}{l}
\{\al(x)+\bt(x)\}\{\lm g(\f y,\f z)+2n\,\eta(y)\eta(z)\} \\[4pt]
+\al(y)\{\lm g(\f x,\f z)+2n\,\eta(x)\eta(z)\}+\al(z)\{\lm g(\f x,\f y)+2n\,\eta(x)\eta(y)\}\\[4pt]
+(\lm+2n)\{ g(x, \f y)\eta(z)+g(x, \f z)\eta(y)\}=0
\end{array}
\end{equation}

Setting successively $x$, $y$ and $z$ as $\xi$ in the latter equality and then combining the obtained equalities, we get 
\begin{equation}\label{albt}
\al=\al(\xi)\eta,\qquad \bt=-2\al(\xi)\eta. 
\end{equation}
Imposing these equalities in \eqref{apRs} gives
\[
\begin{array}{l}
\lm\al(\xi)\{-2g(\f y,\f z)\eta(x)+g(\f x,\f z)\eta(y)+g(\f x,\f y)\eta(z)\} \\[4pt]
+(\lm+2n)\{ g(x, \f y)\eta(z)+g(x, \f z)\eta(y)\}=0,
\end{array}
\]
which for $z=\xi$ implies 
\[
\lm\al(\xi)g(\f x,\f y)+(\lm+2n) g(x, \f y)=0.
\]
The latter equality is fulfilled if and only if $\lm=-2n$ and $\al(\xi)=0$.

Vice versa, let $\M$ be Einstein, i.e. $\rho = 2n g$. Then, equality 
\eqref{apRs-def} takes the following form
\begin{equation}\label{albtggg}
\{\al(x)+\bt(x)\}g(y,z)+\al(y) g(x, z)+\al(z) g(x,y)=0.
\end{equation}
Again by setting successively $x$, $y$ and $z$ as $\xi$ and combining the obtained equalities,
we get the formulas \eqref{albt} and then \eqref{albtggg} implies
\[
\al(\xi)\{-2\eta(x)g(y,z)+\eta(y) g(x, z)+\eta(z) g(x,y)\}=0
\]
for arbitrary $x$, $y$, $z$ and therefore $\al(\xi)=0$ holds.
Thus, \eqref{apRs-def} is satisfied providing $\al(\xi)$ vanishes, which completes the proof.
\end{proof}

A manifold is called \emph{special weakly Ricci symmetric} if its non-vanishing Ricci tensor 
satisfies the following condition  \cite{SinKha01}
\begin{equation}\label{swRs-def}
(\n_x \rho)(y,z) = 2\al(x)\rho(y,z) + \al(y)\rho(x,z) + \al(z)\rho(x,z),
\end{equation}
where $\al$  is a non-vanishing 1-form.

\begin{proposition}\label{prop:7-SletaEl}
Let $\M$ be a Sasaki-like manifold admitting
a Ricci-like soliton with vertical potential.
Then the manifold is special weakly Ricci symmetric  
if and only if the manifold is Einstein.
\end{proposition}
\begin{proof}
Comparing \eqref{swRs-def} with \eqref{apRs-def}, we deduce that 
an almost pseudo Ricci symmetric manifold with $\al=\bt$ is a special weakly Ricci symmetric manifold.
Then, \eqref{albt} implies $\al=0$ and therefore the manifold has $\n\rho=0$. 
By virtue of \propref{prop:6-SletaEl}, we conclude the validity of the statement.
\end{proof}

Finally in the series, for the studied manifold $\M$, let us consider the curvature condition 
$Q\cdot R=0$, which can read as
\begin{equation}\label{QdotR-def}
R(x,y,z,Qw) - R(Qx,y,z,w) - R(x,Qy,z,w) - R(x,y,Qz,w) = 0. 
\end{equation}

\begin{proposition}\label{prop:8-SletaEl}
Let $\M$ be a Sasaki-like manifold with the curvature property $Q\cdot R=0$.
Then the manifold does not admit
a Ricci-like soliton with vertical potential.
\end{proposition}
\begin{proof}
Taking the trace of \eqref{QdotR-def} for $x=e_i$, $w=e_j$ and due to the symmetry of $Q$, we obtain
\[
\rho(Qy,z)+\rho(y,Qz)=0.
\] 
After that we use \eqref{SlRl-v-rho-k=const} and get
\[
\lm^2 g(\f y,\f z)-8n^2 \eta(y)\eta(z)=0,
\]
which has no solution. 
\end{proof}


\section{Parallel symmetric second-order covariant tensor on a Sasaki-like almost contact B-metric manifold }

In \cite{Eis}, L.P. Eisenhart proved that if a positive definite Riemannian manifold admits a second-order parallel
symmetric tensor other than a constant multiple of the metric tensor,
then it is reducible.
In \cite{Levy}, H. Levy proved that a second-order parallel symmetric non-singular (with non-vanishing determinant)
tensor in a space of constant curvature is proportional to the metric tensor.
In \cite{Sha}, R. Sharma proved a generalization over Levy's theorem for dimension greater than two in non-flat real space forms.

Sasaki-like almost contact B-metric manifolds are not need to be real space forms in general.
We now prove an assertion of the kind of the theorems above.

\begin{proposition}\label{prop:h-Sl}
On an arbitrary Sasaki-like manifold, every symmetric
second-order covariant tensor that is parallel with respect to the Levi-Civita connection of the B-metric,
is a constant multiple of this metric.
\end{proposition}

\begin{proof}
Let $h$ be a symmetric 
$(0,2)$-tensor field which is parallel, i.e. $\n h=0$ with respect to the Levi-Civita connection of $g$.
We use the Ricci identity for $h$, i.e.
\[
\left(\n_x\n_y h\right)(z,w)-\left(\n_y\n_x h\right)(z,w)=-h\left(R(x,y)z,\,w\right)-h\left(z,\,R(x,y)w\right).
\]
Since $\n h=0$, the latter identity implies
\[
h\left(R(x,y)z,\,w\right)+h\left(z,\,R(x,y)w\right)=0
\]
and therefore the following property is valid
\begin{equation}\label{hxi}
	h\left(R(x,y)\xi,\,\xi\right)=0.
\end{equation}

Let the considered manifold be Sasaki-like. Then, we replace the expression of $R(x,y)\xi$ from \eqref{curSl} in \eqref{hxi} and
we obtain
\[
h(x,\xi)\eta(y)-h(y,\xi)\eta(x)=0.
\]
If we put $y=\xi$ in the latter equality, it follows
\begin{equation}\label{hxxi-Sl}
	h(x,\xi)=h(\xi,\xi)\eta(x).
\end{equation}
Since $\eta(\n_x\xi)$ vanishes because of $g(\xi,\xi)=1$, then \eqref{hxxi-Sl} implies
\begin{equation}\label{hnxxi-Sl}
	h(\n_x \xi,\xi)=0.
\end{equation}
From the expression of $\left(\n_x h\right)(y,z)$ in the case $\n h=0$ and $y=z=\xi$ we have
\[
x\bigl(h(\xi,\xi)\bigr)=2h(\n_x \xi,\xi),
\]
which means that $h(\xi,\xi)$ is a constant because of \eqref{hnxxi-Sl}.

Next, we take the covariant derivative of \eqref{hxxi-Sl} with respect to $y$ and we get
$h(x,\f y)=h(\xi,\xi)g(x,\f y)$, bearing in mind the first equality of \eqref{curSl}. 
We substitute $y$ for $\f y$
in the latter equality for $h$ and use \eqref{strM} and \eqref{hxxi-Sl} to obtain
the following
\begin{equation*}\label{hxy-Sl}
	h(x,y)=h(\xi,\xi)g(x,y).
\end{equation*}
The latter equality means that $h$ is a constant multiple of $g$ and therefore we proved the following
\end{proof}

Let us apply this assertion to a Ricci-like soliton.

\begin{theorem}\label{thm:h-Sl}
Let $\M$ be a Sasaki-like manifold of dimension $2n+1$ and let $h$ be the tensor $\frac12 \mathcal{L}_{\xi} g  + \rho + \mu\, \g  + \nu\, \eta\otimes \eta$, where $\mu,\nu\in\R$. The tensor $h$ is parallel with respect to $\n$ of $g$ if and only if	 $\M$ admits a Ricci-like soliton with potential $\xi$ and constants $(\lm,\mu,\nu)$, where $\lm=-h(\xi,\xi)=-\mu-\nu-2n$.
\end{theorem}
\begin{proof}
Since $h=\frac12 \mathcal{L}_{\xi} g  + \rho + \mu\, \g  + \nu\, \eta\otimes \eta$
and for Sasaki-like manifolds we have $\left(\mathcal{L}_{\xi} g\right)(x,y)=g(\n_x\xi,y)+g(x,\n_y\xi)=-2g(x,\f y)$, i.e.
$\frac12 \mathcal{L}_{\xi} g=-\g+\eta\otimes \eta$. Then $h$ takes the form
\begin{equation}\label{h-Sl}
	h=\rho + (\mu-1) \g  + (\nu+1) \eta\otimes \eta.
\end{equation}

Now, let $h$ be parallel. 
According to \propref{prop:h-Sl}, $h=h(\xi,\xi)g$ holds true in this case. From \eqref{h-Sl}
and the last equality in \eqref{curSl}, we obtain
$h(\xi,\xi)=2n + \mu + \nu$  and therefore
\(
	h=(2n + \mu + \nu)g.
\) 
The latter equality and \eqref{defRl} imply that there is a Ricci-like soliton with constants $(\lm,\mu,\nu)$,
where $\lm=-\mu-\nu-2n$.

Vice versa, let $\M$ admit a Ricci-like soliton with constants $(\lm,\mu,\nu)$, i.e. \eqref{defRl} is valid.
The latter condition can be rewritten as $h=-\lm g$.
Bearing in mind that $\lm$ is constant and $g$ is parallel, it follows that $h$ is parallel with respect to $\n$, too.
\end{proof}


\section{Example of an Einstein-like Sasaki-like manifold, admitting a Ricci-like soliton 
with potential Reeb vector field}


In Example 2 of \cite{IvMaMa45}, it is given
a Lie group $G$ of
dimension $5$ (i.e. $n=2$)
 with a basis of left-invariant vector fields $\{e_0,\dots, e_{4}\}$
and the corresponding Lie algebra is 
 defined by the commutators
\begin{equation*}\label{comEx1}
\begin{array}{ll}
[e_0,e_1] = p e_2 + e_3 + q e_4,\quad &[e_0,e_2] = - p e_1 -
q e_3 + e_4,\\[4pt]
[e_0,e_3] = - e_1  - q e_2 + p e_4,\quad &[e_0,e_4] = q e_1
- e_2 - p e_3,\qquad p,q\in\R.
\end{array}
\end{equation*}
Then, $G$ is equipped with an almost contact B-metric structure by
\begin{equation*}\label{strEx1}
\begin{array}{l}
g(e_0,e_0)=g(e_1,e_1)=g(e_2,e_2)=-g(e_{3},e_{3})=-g(e_{4},e_{4})=1,
\\[4pt]
g(e_i,e_j)=0,\quad
i,j\in\{0,1,\dots,4\},\; i\neq j,
\\[4pt]
\xi=e_0, \quad \f  e_1=e_{3},\quad  \f e_2=e_{4},\quad \f  e_3=-e_{1},\quad \f  e_4=-e_{2}.
\end{array}
\end{equation*}
There are computed the components of $\n$ for $g$ and the non-zero of them are the following
\begin{equation}\label{nei}
\begin{array}{c}
\begin{array}{ll}
\n_{e_0} e_1 = p\, e_2 + q\, e_4,\qquad & \n_{e_0} e_2 = - p\, e_1 - q\, e_3, \\[4pt]
\n_{e_0} e_3 = - q\, e_2 + p\, e_4,\qquad &  \n_{e_0} e_4 = q\, e_1 - p\, e_3,  
\end{array}
\\
\begin{array}{c}\\[-8pt]
\n_{e_1}e_0 = - e_3,\quad \n_{e_2} e_0 = - e_4,\quad \n_{e_3} e_0 = e_1, \quad \n_{e_4}e_0 = e_2,
\\[4pt]
\n_{e_1}e_3 = \n_{e_2} e_4 = \n_{e_3} e_1 = \n_{e_4}e_2 = - e_0.
\end{array}
\end{array}
\end{equation}
It is verified that the constructed almost contact B-metric manifold
$(G,\f,\allowbreak{}\xi,\allowbreak{}\eta,\allowbreak{}g)$ is Sasaki-like.

The components of the curvature tensor $R_{ijkl}=g(R(e_i,e_j)e_k,e_l)$ 
and those of the Ricci tensor  $\rho_{ij}=\rho(e_i,e_j)$ are computed for the same manifold in \cite{Man62}.
The non-zero of them are determined by the following equalities  and the usual properties $R_{ijkl}=-R_{jikl}=-R_{ijlk}$: 
\begin{equation*}\label{Rex1}
\begin{array}{l}
R_{0110}=R_{0220}=R_{1234}=R_{1432}=R_{2341}=R_{3412}=R_{1331}=R_{2442}=1,\\[4pt]
R_{0330}=R_{0440}=-1,\qquad \rho_{00}=4.
\end{array}
\end{equation*}
Since $\rho=4\,\eta\otimes\eta$ is satisfied, then $(G,\f,\allowbreak{}\xi,\allowbreak{}\eta,\allowbreak{}g)$
is $\eta$-Einstein with constants  
\begin{equation}\label{abcS}
(a,b,c)=(0,0,4).
\end{equation}
Moreover, it is clear that $\tau=\ttt=4$.

Also, there are computed the components $\left(\LL_\xi g\right)_{ij}=\left(\LL_\xi g\right)(e_i,e_j)$ and
the non-zero of them are the following
$
\left(\LL_\xi g\right)_{13}=\left(\LL_\xi g\right)_{24}=\left(\LL_\xi g\right)_{31}=\left(\LL_\xi g\right)_{42}=2,
$ 
which means that $\LL_\xi g=-2\g+2\eta\otimes\eta$. 
 
After that, it is obtained that \eqref{defRl-v} is satisfied for 
\begin{equation}\label{lmnS}
(\lm,\mu,\nu)=(0,1,-5)
\end{equation}
 and therefore $(G,\f,\allowbreak{}\xi,\allowbreak{}\eta,\allowbreak{}g)$
admits a Ricci-like soliton with potential $\xi$. 
Therefore, the condition $\lm+\mu+\nu=-2n$ is fulfilled and this 
is in unison with \thmref{thm:RlSl}.

Now, using \eqref{nei} and the only non-zero component $\rho_{00}=4$ of $\rho$, 
we compute the components of $(\n_i \rho)_{jk}=(\n_{e_i} \rho)(e_j, e_k)$ of $\n\rho$ 
and the non-zero ones of them are the following
\begin{equation*}\label{nrho-ex}
(\n_1 \rho)_{30}=(\n_2 \rho)_{40}=(\n_3 \rho)_{10}=(\n_4 \rho)_{20}=4
\end{equation*}
and their  symmetric about $j$ and $k$. 
These results are in accordance with \propref{prop:rho-Sl}. 

In conclusion, the constructed 5-dimensional example of a Sasaki-like manifold 
with the results in \eqref{abcS} and \eqref{lmnS} 
supports also \thmref{thm:RlSl}, \thmref{thm:k=const} for $k=1$ and 
it agrees with \propref{prop:h-Sl} and 
\thmref{thm:h-Sl} for the trivial case $h=0$.

\subsection*{Acknowledgment}
The author was supported by projects MU19-FMI-020 and FP19-FMI-002 of the Scientific Research Fund,
University of Plovdiv Paisii Hilendarski, Bulgaria.

\end{document}